\documentclass[leqno]{article}
\usepackage[margin=1in]{geometry}
\usepackage{mathtools, amsfonts,amsthm,esdiff,enumitem}
\usepackage[usenames,dvipsnames]{xcolor}
\usepackage[colorlinks=true,linkcolor=NavyBlue,urlcolor=Black,citecolor=Blue]{hyperref}

\usepackage[backend=bibtex,style=alphabetic,sorting=nyt]{biblatex}
\newcommand{\R}{\mathbb{R}}
\newcommand{\N}{\mathbb{N}}
\newcommand{\m}{\mathfrak{m}}
\newcommand{\VF}{\mathfrak{X}}

\newcommand{\Dcal}{\mathcal{D}}
\newcommand{\Ecal}{\mathcal{E}}
\newcommand{\Jcal}{\mathcal{J}}

\newcommand{\Pcal}{\mathcal{P}}

\DeclareMathOperator{\vol}{vol}
\DeclareMathOperator{\Riem}{Riem}
\DeclareMathOperator{\Ric}{Ric}
\DeclareMathOperator{\Exp}{Exp}
\DeclareMathOperator{\Log}{Log}
\DeclareMathOperator{\Tr}{Tr}
\DeclareMathOperator{\Id}{Id}

\DeclareMathOperator{\BM}{\mathsf{BM}}
\DeclareMathOperator{\CD}{\mathsf{CD}}
\DeclareMathOperator{\TBM}{\mathsf{TBM}}
\DeclareMathOperator{\TCD}{\mathsf{TCD}}

\DeclareMathOperator{\Ent}{Ent}

\renewcommand{\epsilon}{\varepsilon}

\newtheorem{theorem}{Theorem}[section]
\newtheorem{lemma}[theorem]{Lemma}
\newtheorem{prop}[theorem]{Proposition}

\theoremstyle{definition}
\newtheorem{defn}{Definition}[section]

\title{Equivalence between the timelike Brunn-Minkowski inequality and timelike Bakry-\'Emery-Ricci lower bound on weighted globally hyperbolic spacetimes} 
\author{Osama Farooqui}
\date{}

\begin{document}
\maketitle
{\hypersetup{linkcolor = Black}
\tableofcontents}
\abstract{We prove the timelike Brunn-Minkowski inequality $\mathsf{TBM}(K,N)$ implies a timelike lower bound on the Bakry-\'Emery-Ricci curvature on weighted globally hyperbolic spacetimes.  This result, together with the well-known equivalence between timelike Bakry-\'Emery-Ricci lower bounds and the $\mathsf{TCD}(K,N)$ condition, and the fact that $\mathsf{TCD}(K,N)$ spaces support the timelike Brunn-Minkowski inequality, draws an equivalence between $\mathsf{TBM}(K,N)$ and $\mathsf{TCD}(K,N)$ in the smooth setting.} 

\section{Introduction}
The purpose of this paper is to investigate the correspondence between Ricci curvature and volume distortion along geodesic interpolation. In particular, our interests lie in an equivalent characterization of Ricci lower bounds in terms of the transport of measures on probability space, especially as this characterization makes sense on more abstract, non-smooth spaces. 

In his 1994 PhD thesis (later published in 1997) McCann \cite{McCann97} introduced the notion of displacement convexity for functionals over the space of probability measures, and showed that the Boltzmann-Shannon entropy $\Ent(\mu) := \int \log(d\mu/d\vol_n)d\mu$. was displacement convex on $\R^n$. Some time after, Otto and Villani \cite{OttoVillani} heuristically showed that, on Riemannian manifolds, Ricci lower bounds imply a kind of convexity of the Boltzmann-Shannon entropy, suitably modified to account for the Ricci lower bound $k$ and the dimension $n$. In particular when $k=0$, the notion coincides with McCann's displacement convexity. The heuristic of Otto and Villani was proved rigorously using techniques of optimal transport by Cordero-Erausquin, McCann, and Schmuckenschl\"ager \cite{BBL} for $\Ric\geq 0$. Sturm and von Renesse \cite{RenesseSturm} later proved that Ricci lower bounds are in fact equivalent to displacement convexity of $\Ent$. 

Following this development, Sturm (using $\Ent$ in \cite{Sturm1} and the Reny\'i entropy $U_N(\mu):= -\int (d\mu/d\m)^{-1/N} d\mu$ in \cite{Sturm2}) used this equivalence as a definition of Ricci curvature on metric measure spaces $(X,d,\m)$. Independently, Lott and Villani \cite{LottVillani} came to the same conclusion, using a much broader family of entropies. Following \cite{Sturm2}, metric measure spaces satisfying the Reny\'i-entropic convexity are said to satisfy the {\it $(K,N)$-curvature-dimension condition}, and are called $\CD(K,N)$ spaces. Here $K$ plays the role of lower bound on curvature, and $N$ an upper bound on the dimension. Those satisfying Boltzmann-Shannon-entropic convexity as said to be {\it entropic $\CD(K,N)$}, or $\CD^e(K,N)$. On smooth manifolds the notions are the same; we make the distinction here only to give historical context. We remark that there are many related kinds of curvature dimension conditions, which can be obtained by considering different kinds of convexity for different kinds of entropy. We refer generally to this family as $\CD$ spaces. 

Much can be said about $\CD$ spaces. In particular, and relevant to our interests, is the existence of many well known geometric and functional inequalities \cite[Ch.\ 18--21]{Villani09}. This includes the Brunn-Minkowski inequality, suitably modified to account for the dimension and curvature bounds.

Very recently, Magnabosco, Portinale, and Rossi \cite{MPR1, MPR2} proved in a pair of papers  that weighted Riemannian manifolds that support the Brunn-Minkowski inequality also support Ricci lower bounds. That is, the Brunn-Minkowski condition $\BM(K,N)$ is in fact equivalent to the $\CD(K,N)$ condition in the smooth setting. In the non-smooth essentially non-branching setting, they showed that a stronger form of $\BM(K,N)$ is equivalent to $\CD(K,N)$.

The entirety of the discussion so far has been concerned with Riemannian geometry, but Ricci curvature is of interest also in mathematical relativity and Lorentzian geometry. Via Einstein's equations, Ricci curvature encodes the matter and energy content of spacetime. Thus, on physical grounds, the bulk of study is concentrated on {\it timelike} Ricci lower bounds, which encapsulate lower bounds on energy density; they appear for example in the energy conditions of Hawking and Penrose (see \cite{Steinbauer2022} and references therein for an overview). 

Motivated by this, and a potential thermodynamical connection, McCann \cite{McCann2020}, and independently, Mondino and Suhr \cite{MondinoSuhr}, developed a theory of optimal transport on Lorentzian manifolds, proved an equivalence between timelike Ricci lower bounds and timelike entropic convexity, and thus introduced, in the smooth setting, the {\it timelike entropic (K,N)-curvature-dimension condition} $\TCD^e(K,N)$. 

Cavalletti and Mondino \cite{CavallettiMondino} were the first to generalize these results to non-smooth {\it Lorentzian length spaces}, a Lorentzian analog of metric measure spaces developed some years prior by Kunzinger and S\"amann \cite{KunzingerSamann}. On $\TCD^e(K,N)$ spaces, they obtained a kind of Brunn-Minkowski inequality, albeit with weaker coefficients. Later, Braun \cite{Braun} developed the $\TCD(K,N)$ condition on Lorentzian length spaces (using the Reny\'i entropy) and obtained a stronger (timelike) Brunn-Minkowski property, which we refer to as $\TBM(K,N)$. 

The path forward is clear: following in the footsteps of Magnabosco, Portinale, and Rossi, we will prove an equivalence between $\TCD(K,N)$ and $\TBM(K,N)$ in the smooth setting. Precisely we prove, by contrapositive, the following:
\begin{theorem}
Let $(M,g, \m)$ be a weighted \textup{(eq.\ \eqref{weighted})} globally hyperbolic spacetime of dimension $n$. Suppose for some $K\in \R$, $N>1$, $M\in \TBM(K,N)$ \textup{(defn. \ref{TBM})}. Then $\Ric^{N,\m}\geq K$ \textup{(eq.\ \eqref{bakryemeryricci})}. 
\end{theorem} 

The proof follows its Riemannian analog in \cite{MPR1}, necessarily adapted to the Lorentzian setting.

\section{Preliminaries}
\subsection{Geodesics, curvature, and volume}\label{geometrysection}
Let $(M,g)$ be a $C^2$-smooth Lorentzian manifold of dimension $n\in\N$ with signature $(+,-,\dots,-)$ Note that on any smooth manifold there is a Riemannian metric which induces the manifold topology. The metric $g$ induces a Levi-Civita connection, from which all other relevant geometric quantities arise.  We use the symbol $D$ to denote covariant differentiation with respect to this connection, and reserve $\nabla$ to denote the gradient. In an abuse of notation, we use the same symbol, $D$, to denote exterior differentiation. We hope it is clear from context which definition we mean.  

For any $x\in M$ we may define the {\it exponential map at $x$ }, $\exp_x:T_xM\to M$, which is the solution map to the initial value geodesic problem 
\begin{align}\label{geodesicivp}
\begin{cases}
D_{\dot \gamma}\dot\gamma=0;\\
\gamma_0=x;\\
\dot\gamma_0=v,
\end{cases}
\end{align}
for some specified $v\in T_{x}M$. That is, $\gamma_t\coloneqq \exp_x(tv)$ is a short time solution to \eqref{geodesicivp}, called an {\it (affinely parametrized) geodesic}. For us the term {\it geodesic} will always mean a geodesic affinely parametrized along $[0,1]$. The map $\exp_x$ is locally a diffeomorphism, and its local inverse is denoted by $\log_x$. As a local diffeomorphism, $\exp_x$ can be used a chart near the point $x$, once a choice of isomorphism $\phi: T_xM\cong \R^n$ is made. An open set $U\ni x$ is a {\it normal neighbourhood} about $x$ if $\phi\circ \log_x :U\to V\subset \R^n$ is a diffeomorphism. The charts $\phi \circ \log_x$ are referred to as {\it normal coordinates}. A normal neighbourhood is convex if it is a normal neighbourhood about $x$ for every $x\in U$. 

In a normal neighbourhood $U\subset TM$ about any point $(x,0)$, we may introduce the {\it exponential map} $\exp:U\to M$ as 
\[\exp(x,v)\coloneqq \exp_x(v),\]
as well as the map $\Exp:U\to M\times M$, given by 

\begin{align}
\Exp(x,v)\coloneqq (x, \exp_x(v)).
\end{align}
In a small neighbourhood about any point $(x,0)\in TM$, $\Exp$ is well defined and a local diffeomorphism. Its inverse is denoted by $\Log$. It suffices for our purposes to define $\Log$ in a neighbourhood of some special point $(x_0, 0)\in TM$, defined in section \ref{Resultsection}.

$\Log(x,y)$ can be understood as the vector on $TM$ which `points $x$ in the direction of $y$'. On $\R^n$, $\Log(x,y)=(x,y-x)$. It can also be understood via the boundary value geodesic problem 
\begin{align}\label{geodesicbvp}
\begin{cases}
D_{\dot \gamma}\dot\gamma=0;\\
\gamma_0=x;\\
\gamma_1=y.
\end{cases}
\end{align}
Whenever a solution to \eqref{geodesicbvp} exists and is unique, it is given by $\gamma_t=F_t(x,y)\coloneqq \exp(t\Log(x,y))= \exp_x(t\log_x(y))$. In this sense, the domain of $\Log$ consists of all pairs of points $(x,y)$ such that there exists a unique geodesic joining $x$ to $y$.

We denote the set of $C^1$ vector fields on $M$ by $\VF(M)$. The connection induces the {\it Riemann tensor} $\Riem:\VF(M)\times \VF(M)\times \VF(M)\to \VF(M)$ defined by 
\begin{align*}
\Riem(X,Y)Z=[D_X,D_Y]Z-D_{[X,Y]}Z.
\end{align*}

The Riemann tensor captures the curvature data of the manifold. One way to extract this data is through {\it Jacobi fields} -- vector fields $J(t)$ defined over a geodesic $\gamma_t$ which solve the linear second order ODE
\begin{align}
\begin{cases}
D_{\dot\gamma}D_{\dot \gamma} J-\Riem(J, \dot\gamma)\dot\gamma=0;\\
J(0)=v;\\
\dot J(0)=w,
\end{cases}
\end{align}
for some initial pair of vectors $(v,w)\in T_{(\gamma_0, \dot\gamma_0)}TM$. It is a remarkable fact that the solution map of this initial value system is $(D\exp)_{(\gamma, \dot \gamma)}$, i.e $J(t)=  (D\exp)_{(\gamma, \dot \gamma)}(v, tw)$. 
Likewise, recalling the map $F_t(x,y)= \exp(t\Log(x,y))$, $DF_t$ is a solution map of the {\it boundary value} Jacobi equation; $J(t)=(DF_t)_{(x,y)}(v,w)$ is a Jacobi field along the geodesic $F_t(x,y)$ which solves 
\begin{align}
\begin{cases}
D_{\dot F_t(x,y)}D_{\dot F_t(x,y)} J-\Riem(J, \dot F_t(x,y))\dot F_t(x,y)=0;\\
J(0)=v;\\
J(1)=w.
\end{cases}
\end{align}

Next, by taking traces of the Riemann tensor, one defines the {\it Ricci tensor} $\Ric:\VF(M)\times \VF(M)\to \R$ as
\[\Ric(X,Y)\coloneqq \Tr[ Z\mapsto \Riem(X,Z)Y].\]

This is a symmetric linear operator, and is therefore completely defined by its quadratic form $X\mapsto \Ric(X,X)$. We say $M$ has a Ricci lower bound $K\in\R$ if the quadratic form is bounded from below by $K$, i.e. if 
\[\Ric(X,X)\geq  Kg(X,X)\quad \forall X\in \VF(M).\]

The metric $g$ endows $M$ with a natural measure $\vol_g$, which in coordinates is of the form $d\vol_g= |\det g|^{1/2}dx$. We say $M$ is {\it weighted} if it is equipped with a measure $\m$ absolutely continuous with respect to $\vol_g$. By convention, we write its density as
\begin{align}\label{weighted}
d\m=e^{-\psi}d\vol_g,
\end{align}
for some $\psi\in C^2(M)$. On a weighted $n$-manifold, and for any $N\neq n$, we may define the {\it $N$-Bakry-\'Emery-Ricci tensor} $\Ric^{N,\m}$ as 
\begin{align}\label{bakryemeryricci}
\Ric^{N,\m}\coloneqq \Ric+D^2\psi-\frac{D\psi\otimes D\psi}{N-n}. 
\end{align}
In the case $N=n$, we assume $\psi\equiv0$, so that $\Ric^{n,\m}= \Ric$. 
\subsection{The causal structure of Lorentzian manifolds}\label{Lorentziansection}

Due to the indefinite inner product induced by $g$ on the tangent bundle $TM$, vectors on $M$ are of three types. We say that $v\in TM$ is:
\begin{enumerate}
\item {\it timelike} if $g(v,v)>0$; 
\item {\it lightlike} if $g(v,v)=0$;
\item {\it spacelike} if $g(v,v)<0$.
\end{enumerate}
At each point $x\in M$, the set of timelike vectors on $T_xM$ has two connected components. WLOG we denote one component as the {\it future direction} of $T_xM$, and the other as the {\it past direction}. $M$ is said to be time orientable if there is a smooth choice of future direction on the whole bundle $TM$.  A {\it spacetime} is any smooth, connected, time orientable Lorentzian manifold.

A $C^1$ curve $\gamma_t$ is called timelike if its tangent vector $\dot\gamma_t$ is timelike for all $t$. Similarly it may be called lightlike or spacelike whenever $\dot\gamma_t$ is. A curve is {\it causal} if it is non-spacelike at every point $t$. Timelike and causal curves can also be given the designation {\it future-directed} or {\it past-directed} whenever $\dot\gamma_t$ is in the future direction or past direction respectively. 

Such curves give a notion of causal order on the spacetime. We say a point $y\in M$ is in the causal future of $x$ if there is a future-directed causal curve which begins at $x$ and passes through $y$. We write $x\leq y$ whenever this occurs. Similarly we say that $y$ is in the timelike future of $x$, and write $x\ll y$ whenever the joining curve is timelike.

On a spacetime $M$ we may define the Lagrangian 
\begin{align}
L_q(v)\coloneqq \begin{cases}
\frac1qg^q(v,v) & \text{$v$ is future directed causal};\\[2mm]
-\infty & \text{otherwise},
\end{cases}
\end{align}
for $0<q< 1$. The associated action on Lipschitz curves $\gamma_t\in C^{0,1}([0,1],M)$ is then
\begin{align}
A_q[\gamma_t]\coloneqq \int_0^1 L_q(\dot\gamma_t)dt.
\end{align}
Define then the map $\ell(x,y;q)$ given by extremizing the action over Lipschitz curves with fixed endpoints.
\begin{align}\label{ell}
\ell(x,y; q)\coloneqq \sup\left\{A_q[\gamma_t]: \gamma\in C^{0,1}([0,1],M), \gamma_0=x, \gamma_1=y\right\},
\end{align}
with convention $(-\infty)^r=-\infty$ for all $r>0$. It happens that the map $\ell(x,y)\coloneqq (q \ell(x,y;q))^q$ is independent of $q$; we call $\ell(x,y)$ the {\it time separation function}.

The time separation function serves as an analogue of the metric distance in the Riemannian setting, while also respecting the causal structure of $M$. In the language of general relativity, $\ell(x,y)$ measures the maximal proper time that may elapse in journeying from event $x$ to $y$. It can be seen that $x\leq y\iff \ell(x,y)\geq 0$ and $x\ll y\iff \ell(x,y)>0$.   If a timelike curve $\gamma_t$ realizing the supremum in \eqref{ell} exists, then it must be a geodesic (affinely parametrized along $[0,1]$). One can then conclude that, if $x,y$ are in the domain of $\Log$, then 
\[\ell(x,y)=|\dot\gamma_t|_g=|\Log(x,y)|_g=|\log_x(y)|_g.\]

One can guarantee that the supremum in \eqref{ell} is attained under the assumption of {\it global hyperbolicity} of $M$. A spacetime $M$ is globally hyperbolic if $M$ does not contain closed causal curves, and the sets
\begin{align}\label{GlobalHyperbolicity}
J(x,y)\coloneqq \{z\in M: x\leq z\leq y\},
\end{align}
are compact. 
Global hyperbolicity implies timelike geodesy i.e. for any $x\ll y$, there exists a timelike geodesic $\gamma_t$ joining $x$ to $y$. Notice that the geodesic may not be unique, hence $(x,y)$ may not be in the domain of $\Log$. 

The time separation $\ell(x,y)$ is smooth precisely on the complement of the singular set $\operatorname{sing}(\ell)$; that is, it is smooth on the set of all pairs $(x,y)\in M\times M$ such that $\ell(x,y)>0$, and both $x$ and $y$ lie on the relative interior of such a maximizing geodesic described above in \eqref{ell} \cite[Defn.\ 2.1, Thm.\ 3.6]{McCann2020}.

We remark here that the term `time separation function' carries different definitions in the literature. It may also refer to the function $\tau$ (as in \cite{CavallettiMondino}) which may be defined, with respect to what was introduced above, as 
\[\tau(x,y)\coloneqq \ell^+(x,y)\coloneqq \max\{\ell(x,y), 0\}.\]
The function $\tau$ measures time separation just as $\ell$, but cannot distinguish between lightlike and spacelike points. Nevertheless it can be used in replacement of $\ell$ in everything that follows, mutatis mutandis, and in fact will be used in the definition of the timelike Brunn-Minkowski inequality. We will however refer to it by the symbol $\ell^+$ so as not the confuse the reader with the {\it distortion coefficient}, also labelled $\tau$, defined below.

\subsection{Optimal transport on globally hyperbolic spacetimes}\label{OTsection}
We denote by $\Pcal(M)$ the set of Borel probability measures on $M$. On $\Pcal(M)$ we may define the {\it q-Lorentz-Wasserstein time separation} for any $0<q< 1$:
\begin{align}\label{ell_qdefn}
\ell_q(\mu,\nu)\coloneqq \left(\sup_{\gamma \in \Gamma(\mu,\nu)}\int_{M\times M} \ell^q(x,y)d\gamma(x,y)\right)^{1/q},
\end{align}
where $\Gamma(\mu,\nu)$ denotes the set of admissible couplings, 
\[\Gamma(\mu,\nu)\coloneqq\{\gamma\in \Pcal(M\times M): \pi^1_\# \gamma=\mu, \pi^2_\# \gamma=\nu\}.\]

The maps $\pi^i$ are the projection maps $M\times M\to M$ onto the $i$-th component, for $i=1,2$. 
We say that the pair $(\mu,\nu)$ are {\it timelike q-dualizable} if the following hold: 
\begin{enumerate}
\item $\ell_q(\mu,\nu)\neq -\infty$;
\item There exists lower-semicontinuous $a,b:M\to \R$ with $a\oplus b\in L^1(\mu\otimes \nu)$ and $\ell^q\leq a\oplus b$ for all $(x,y)\in \operatorname{spt}(\mu)\times \operatorname{spt}(\nu)$;
\item There exists $\gamma \in \Gamma (\mu,\nu)$ which is concentrated on $\ell^{-1}(0,\infty)$ and attains the supremum in \eqref{ell_qdefn}.
\end{enumerate}

We say a pair of (measurable) sets $A,B\subset M$ are timelike q-dualizable if their uniform measures $\mu=\frac{\m \vert_{A}}{\m(A)}$ and $\nu=\frac{\m\vert_B}{\m(B)}$ are.

\subsection{The timelike Brunn-Minkowski inequality}

For any $k\in \R$, the {\it generalized sine function} $\sin_k(t)$ is the solution to the initial value problem 
\begin{align}\label{generalizedsine}
\begin{cases}
\ddot f(t)+k f(t)=0;\\
f(0)=0;\\
\dot f(0)=1.
\end{cases}
\end{align}

Consider now the boundary value problem 
\begin{align}\label{generalizedsinebvp}
\begin{cases}
\ddot f(t)+k\theta^2 f(t)=0;\\
f(0)=0;\\
f(1)=1.
\end{cases}
\end{align}

When $k\theta^2<\pi^2$, \eqref{generalizedsinebvp} has a unique non-negative solution, positive on $(0,1]$, given by $\frac{\sin_k(t\theta)}{\sin_k(\theta)}$. We define the {\it reduced distortion coefficient} $\sigma_k^{(t)}(\theta)$ as
\begin{align}\label{distortioncoeff}
\sigma_k^{(t)}(\theta)\coloneqq \begin{cases}
\frac{\sin_k(t\theta )}{\sin_k(\theta)} & k\theta^2<\pi^2; \\
+\infty & \text{otherwise}.
\end{cases}
\end{align}

The function $\sigma_k^{(t)}(\theta)$ appears naturally in the integral version of the differential inequality
\begin{align*}
\ddot f(t)+kf(t)\geq 0.
\end{align*} 

In fact we have the following theorem \cites[Lemma 3.2]{MPR1}[cf.][Thm.\ 14.28]{Villani09}, which in the case $k=0$ gives two well-known characterizations of convexity:
\begin{theorem} \label{theorem1}
Let $f(t):\R\to [0,\infty)$ be $C^2$-smooth. Then the following are equivalent: 
\begin{enumerate}
\item $\ddot f+k f\geq 0$;
\item  for all $a\leq b$: 
$$f((1-t)a+tb)\leq \sigma_{k}^{(1-t)}(b-a)f(a)+\sigma_{k}^{(t)}(b-a)f(b)\quad \forall  \ 0\leq t\leq 1.$$ 
\end{enumerate}
\end{theorem}

Our focus is on inequalities of the form $N\ddot f+Kf\geq 0$, where $K\in \R$ and $N>1$. Thus we will consider reduced distortion coefficients of the form $\sigma_{K/N}$, from which we may define the {\it distortion coefficient} $\tau_{K,N}^{(t)}(\theta)$ as
\begin{align}
\tau_{K,N}^{(t)}(\theta)\coloneqq t^{1/N}\left(\sigma_{\frac{K}{N-1}}^{(t)}(\theta)\right)^{1-1/N},
\end{align}
with the convention $\infty^m=\infty=m\cdot \infty$, for any $m>0$.
It is notable that the map $\sigma(k)\coloneqq \sigma_k^{(t)}(1)$ is log-convex on the domain $k<\pi^2$ for every fixed $t$ \cite[Lemma 1.2]{Sturm2}. As a result, we have that 
\begin{align}\label{distortionineq}
\sigma_{K/N}^{(t)}(\theta)\leq \Big(\sigma_{\frac{K-\kappa}{N-\nu}}^{(t)}(\theta) \Big)^{1-\frac \nu N}\Big(\sigma_{\frac\kappa\nu}^{(t)}(\theta)\Big)^{\frac \nu N}\qquad \forall \ K,\kappa\in \R, \ \forall \ N>\nu>0.
\end{align}
Choosing $\kappa=0$ and $\nu=1$, we establish the simple but useful inequality
\begin{align}\label{tauineq}
\sigma_{K/N}^{(t)}(\theta)\leq \tau_{K,N}^{(t)}(\theta) \quad \forall \ K\in \R, N>1. 
\end{align}

The distortion coefficient comes about from a more refined analysis of the function $f$ in Theorem \ref{theorem1}. If $f:\R\to (0,\infty)$ can be decomposed as $f= f_\|f_\perp$, where $( f_\|^{1/n_1})''+\frac{k_1}{n_1} f_\|^{1/n_1}\leq 0$ and $(f_\perp^{1/n_2})''+\frac{k_2}{n_2}f_\perp^{1/n_2}\leq 0$, then one can check, by applying \cite[Thm.\ 14.28]{Villani09} to $f_\|$ and $f_\perp$, and Holder's inequality, that
\begin{align} \label{mixeddistortionineq}
f^{1/N}((1-t)a+tb)\geq \Big(\sigma_{\frac{k_1}{n_1}}^{(1-t)}(\theta) \Big)^{n_1/N}\Big(\sigma_{\frac {k_2}{n_2}}^{(1-t)}(\theta)\Big)^{n_2/N} f^{1/N}(a)+\Big(\sigma_{\frac{k_1}{n_1}}^{(t)}(\theta) \Big)^{n_1/N}\Big(\sigma_{\frac{k_2}{n_2}}^{(t)}(\theta)\Big)^{n_2/N}f^{1/N}(b), 
\end{align}
where $\theta =|b-a|$, $N= n_1+n_2$. In light of \eqref{distortionineq}, the inequality \eqref{mixeddistortionineq} is stronger than what \cite[Thm.\ 14.28]{Villani09} could have deduced from $(f^{1/N})''+\frac{K}{N} f^{1/N}\leq 0$, where $K=k_1+k_2$. When $N>1$, $k_1=0$, and $n_1=1$, we obtain 
\[f^{1/N}((1-t)a+tb)\geq \tau_{K,N}^{(1-t)}(b-a)f^{1/N}(a)+\tau_{K,N}^{(t)}(b-a)f^{1/N}(b).\]

We denote the set of timelike geodesics (affinely parametrized along $[0,1]$) by $\operatorname{TGeo}(M)$. 
Given two sets $A,B\subset M$, the timelike geodesic interpolation $G_t(A,B)$ between $A$ and $B$ is defined by 
\begin{align}\label{geodesicinterpolant}
G_t(A,B)\coloneqq\{\gamma_t: \gamma\in \operatorname{TGeo}(M), \gamma_0\in A, \gamma_1\in B\}.
\end{align}

When $A$ and $B$ are compact, $G_t(A,B)$ is also compact \cite[Lemma 2.5]{McCann2020}, hence $\m$-measurable. In light of Theorem \ref{theorem1}, the timelike Brunn-Minkowski inequality $\TBM(K,N)$ is roughly analogous to the statement that the function $f(t)
\coloneqq \m (G_t(A,B))$ 
decomposes into $f= f_\|f_\perp$ in the manner discussed above.

\begin{defn}[Timelike Brunn-Minkowski Inequality]\label{TBM}
Let $(M,g,\m)$ be a globally hyperbolic spacetime equipped with a measure $\m\ll \vol_g$. Let $N>1$ and $K\in \R$. For any $A,B\subset M$ non-empty and compact, define
\[\Theta=\Theta_K(A,B)\coloneqq \begin{cases}
\displaystyle\inf_{A\times B}\{\ell^+(x,y)\} & K\geq 0;\\[2mm]
\displaystyle\sup_{A\times B}\{\ell^+(x,y)\} & K< 0.
\end{cases}\] 
We say $M\in \TBM(K,N)$ if for any compact connected timelike q-dualizable sets $A,B\subset M$ with non-empty interior, we have
\begin{align}
\m^{1/N}(G_t(A,B))\geq \tau_{K,N}^{(1-t)}(\Theta)\m^{1/N}(A)+\tau_{K,N}^{(t)}(\Theta)\m^{1/N}(B).
\end{align}
\end{defn}

\newpage
\section{Timelike Brunn-Minkowski implies timelike Ricci lower bound}\label{Resultsection}

Let $(M,g,\m)$ be a $C^2$-smooth globally hyperbolic spacetime of dimension $n$, with $d\m=e^{-\psi}d\vol_g$ for some $\psi\in C^2(M)$. Let $N>1$. Suppose by way of contradiction that there exists $\epsilon>0$, $x_0\in M$ and timelike unit vector $v_0\in T_{x_0}M$ such that $\Ric_{x_0}^{N,\m}(v_0,v_0)\leq K-2\epsilon$. We want to then show that $\TBM(K,N)$ fails. That is, we want to find a pair of timelike $q$-dualizable sets $A$ and $B$ such that
\[\m^{1/N}(G_t(A,B))< \tau_{K,N}^{(1-t)}(\Theta)\m^{1/N}(A)+\tau_{K,N}^{(t)}(\Theta)\m^{1/N}(B),\]
for some $0<t<1$. We will choose as $A$ a small `cube' of side length $\delta$ about $x_0$, and as $B$ the image of $A$ under a map $T_\lambda(x):=\exp_{x}(\lambda V(x))$, where $V$ is some $C^1$ vector field, which in particular maps $x_0$ to $v_0$. The fact that $A$ and $B$ will be timelike q-dualizable is clear from the continuity of the time separation function $\ell^+$ and of $T_\lambda(x)$. Indeed, continuity of $\ell^+(x,T_\lambda(y))$ and the fact that $\ell^+(x_0,T_\lambda(x_0))=\lambda|v_0|_g>0$ means that there is an open set $U\ni(x_0,x_0)$ such that $\ell^+(x,T_\lambda(y))>0$ on $U$. By continuity of $T_\lambda$, we can then find $\delta$ such that $A_\delta\times T_\lambda(A_\delta)=A\times B\subset U$. Since $\ell^+(x,y)>0$ for all $(x,y)\in A\times B$, $A$ and $B$ are timelike $q$-dualizable \cite[Lemma 4.4]{McCann2020}. Moreover both $A$ and $B$ can be taken to be in the interior domain of $\Log$ (about $(x_0, 0)$), and therefore any pair $(x,y)\in A\times B$ is in the interior of a affinely parametrized maximizing timelike geodesic. Hence $A\times B\subset \operatorname{sing}(\ell)^\mathsf{c}$, so $\ell$ is smooth on $A\times B$.  All that is left is to prove the following:

\begin{theorem}\label{maintheorem}
Assume there exists $\epsilon>0$ and timelike unit vector $(x_0,v_0)\in TM$ such that $\Ric_{x_0}(v_0,v_0)\leq K-2\epsilon$. Let $\delta>0$, and let $Q_\delta$ be the cube of sidelength $\delta$ spanned by the eigenvectors of the linear map $R_{v_0}\coloneqq \Riem(v_0, \cdot)v_0:T_{x_0}M\to T_{x_0}M$. Let $A_\delta=\exp_{x_0}(Q_\delta)$. There exists $t>0$, and a map $T_\lambda(x)$ such that, for all sufficiently small $\lambda$ and $\delta$, 
\begin{align}\label{maininequality}
\m^{1/N}(G_t(A_\delta,T_\lambda(A_\delta)))< \tau_{K,N}^{(1-t)}(\Theta)\m^{1/N}(A_\delta)+\tau_{K,N}^{(t)}(\Theta)\m^{1/N}(T_\lambda(A_\delta)).
\end{align}
\end{theorem}

Since $R_{v_0}$ is symmetric with respect to the Lorentzian metric and admits the timelike eigenvector $v_0$, it admits $n$ independent eigenvectors, so the set $A_\delta$ is well defined.

The idea of the proof is to compare geodesic interpolation to optimal interpolation. The geodesic interpolation map $F_t(x,y)$ is given by 
\[F_t(x,y)=\exp(t\Log(x,y)).\]
$F_t$ maps a pair of points $(x,y)$ to the geodesic $t$-midpoint between them. In particular, for any $A\times B\subset M\times M$ contained in the domain of $\Log$, $F_t(A\times B)$ is the set of all geodesic $t$-midpoints between any pair of points $(x,y)\in A\times B$. This is related to $G_t(A,B)$, except that the former does not distinguish between timelike and non-timelike geodesics. It therefore follows that $G_t(A,B)\subset F_t(A\times B)$ (in fact, for $A_\delta$ and $B_\delta=T_\lambda(A_\delta)$ as defined above, we have equality of these two sets). 

Geodesic interpolation between $A_\delta$ and $T_\lambda(A_\delta)$ is given by $F_t(A_\delta\times T_\lambda(A_\delta))$. On the other hand, optimal interpolation between these two sets is given by the image of $A_\delta$ under the map $F_t(x, T_\lambda (x))=T_{\lambda t}(x)$. The interpolation is {\it optimal} in the sense that it,  if $\mu_0=\frac{\m\vert_A}{\m(A)}$ and $\mu_1=\frac{\m\vert_B}{\m(B)}$, then $(\Id\times T_\lambda)_\#\mu_0$ attains the supremum \eqref{ell_qdefn} in the $q$-Lorentz-Wasserstein time separation between $\mu_0$ and $\mu_1$. In particular, it is optimal in the sense that $T_{\lambda t}$ transports $A_\delta$ to $T_{\lambda}(A_\delta)$ by sending each point in $A_\delta$ along a single geodesic to to a point in $T_\lambda(A_\delta)$; on the other hand, the geodesic interpolation sends each point in $A_\delta$ along very many geodesics, one for every point in $T_\lambda(A_\delta)$.

Our objective is to obtain an inequality like \eqref{maininequality} for the geodesic interpolation; to do so, we obtain a similar inequality for the optimal interpolation by leveraging Theorem \ref{theorem1}, and then analyze the deviation of one from the other to obtain the correct inequality for the geodesic interpolation.

\subsection{Volume estimate along optimal transport}
The main content of this section is Lemma \ref{InteqLemma}, an inequality similar to \eqref{maininequality} with respect to $T_{\lambda t}(A_\delta)$. For any $x\in M$, and any basis $\{e_i\}$ of $T_xM$, let $A_\delta$ be the image of the cube of side length $\delta$ spanned by this basis under the exponential map $\exp_x$. The infinitesimal volume distortion of $A_\delta$ by the optimal interpolation $T_{\lambda t}$ is given by 
\[\Jcal_x(t)\coloneqq \lim_{\delta\to 0}\frac{\m (T_{\lambda t}(A_\delta))}{\m(A_\delta)}.\]

The quantity of interest is the $N$-th root of this distortion,
\[\Dcal_x(t)\coloneqq \Jcal_x(t)^{1/N}.\]

We remark that the specfic form of $A_\delta$ is not relevant to the results of this section. Any other family of sets $\{A_\delta\}$ will give the same $\Dcal_x(t)$ so long as they shrink to a point $x$, meaning there exists $r>0$ such that $A_\delta\subset B_{\delta r}(x)$ for all $\delta$, and have {\it bounded eccentricity}, meaning there exists some $c>0$ such that $\vol_g(A_\delta)> c \vol_g(B_\delta(x_0))$ for all $\delta$. We will refer to such families as {\it nice}. 

First we show that infinitesimal volume distortion satisfies a differential convexity condition consistent with Theorem \ref{theorem1}.
\begin{lemma}\label{DeqLemma}
Suppose there exists $\epsilon>0$ and timelike unit vector $(x_0,v_0)\in TM$ such that $\Ric_{x_0}(v_0,v_0)\leq K-2\epsilon$. There exists a $C^1$ vector field $V$ defined in a neighbourhood of $x_0$, such that the associated transport map $T_{\lambda t}(x)\coloneqq \exp(\lambda tV(x))$ induces a volume distortion $\Dcal_{x_0}(t)$ which, for sufficiently small $\lambda$, satisfies
\begin{align}\label{Deq2}
\ddot \Dcal_{x_0}(t) +\frac{(K-\epsilon)}{N}\lambda^2\Dcal_{x_0}(t)\geq 0 \quad \forall t\in[0,1].
\end{align}
\end{lemma}
\begin{proof}
Recall that $d\m=e^{-\psi}d\vol_g$. There exists $\phi\in C^2(U)$ in a neighbourhood $U$ of $x_0$ satisfying 
\begin{align}\label{potentialconditions}
\begin{cases}
\nabla\phi\vert_{x_0}=v_0;\\
D\nabla\phi\vert_{x_0}=-\frac{1}{N-n}(D\psi)_{x_0}(v_0)\Id.
\end{cases}
\end{align}
Let $V\coloneqq  \nabla \phi$, $T_\lambda(x)\coloneqq \exp(\lambda V(x))$. 

It follows from Lebesgue differentiation and change of variables that 
\[\Jcal_x(t)= \frac{e^{-\psi(T_{\lambda t}(x))}}{e^{-\psi(x)}}\lim_{\delta\to 0}\frac{\vol_g(T_{\lambda t}(A_\delta))}{\vol_g(A_\delta)}=\frac{e^{-\psi(T_{\lambda t}(x))}}{e^{-\psi(x)}} \det (DT_{\lambda t})_{x}.\]

Let $M_x(t):= (DT_{t})_{x}$. By definition $M_x(t)$ is a matrix of Jacobi fields for each $x\in M$, and therefore solves the system 
\begin{align}\label{JacobiMatrix}
\begin{cases}
\ddot M+R_VM=0; \\
M(0)=\Id;\\
\dot M(0)= DV,
\end{cases}\end{align}
where $R_V\coloneqq\Riem(V, \cdot )V$. Notice that $\Tr R_V=\Ric(V,V)$. Since $M_x(0)=\Id$ is invertible, there exists small $t$ such that $M_x(t)$ is invertible. In this small $t$ regime, we may define $L_x(t)\coloneqq \dot M_x(t)M_x^{-1}(t)$.  The quantity $L_x(t)$ solves the Riccati system 
\begin{align}\label{JacobiLogDeriv}
\begin{cases}
\dot L+ R_V+L^2=0;\\
L(0)= DV.
\end{cases}
\end{align}
Using \eqref{JacobiLogDeriv} and  the Jacobi equation $\diff{}{t}\det M(t)=\det M\Tr(\dot M M^{-1})= \det M(t)\Tr(L(t))$,  one can take derivatives of $\Dcal_{x}(t) = \frac{e^{-\psi(T_{\lambda t}(x))/N}}{e^{-\psi(x)/N}} M^{1/N}_x(\lambda t)$ and find that, for sufficiently small $\lambda$:
\begin{align*}
\begin{split}
\ddot \Dcal_x(t)=& \frac 1N\left[\frac1N\Dcal_x(t)\Big(\Tr(\lambda L_x(\lambda t))-\lambda D\psi(V)\Big)^2- \Dcal_x(t)\Big(\Tr(\lambda^2 L_x^2(\lambda t))+\lambda^2\Ric(V,V)+\lambda^2D^2\psi(V,V)\Big)\right].
\end{split}
\end{align*}
Dividing by $\Dcal_{x}(t)$ and rearranging, we find 
\begin{equation}\label{Deq1}
\frac{\ddot \Dcal_x(t)}{\Dcal_x(t)}=-\frac{\lambda^2}{N}\Ric^{N,\m}(V(x),V(x))+\frac{\lambda^2}{N}\Ecal_x(\lambda t),
\end{equation}
where 
\[\Ecal_x(t)\coloneqq\frac1n \Tr^2\big(L_x(t)\big)-\Tr\big(L_x^2(t)\big)-\frac{n}{N(N-n)}\left(D\psi(V)+\frac{N-n}{n}\Tr\big(L_x(t)\big)\right)^2.\]

Since $V=\nabla\phi$, and  $\phi$ satisfies \eqref{potentialconditions},  $V$ satisfies 
\begin{align}
\begin{cases}\label{Vconditions}
V(x_0)=v_0;\\
DV(x_0)=-\frac{1}{N-n}D\psi(v_0)\Id\eqqcolon \alpha \Id.
\end{cases}
\end{align} 
Therefore at $x=x_0$, \eqref{JacobiLogDeriv} reads
\begin{align}\label{JacobiLogDeriv2}
\begin{cases}
\dot L_{x_0}+R_{v_0}+L_{x_0}^2=0;\\
L_{x_0}(0)=\alpha \Id.
\end{cases}
\end{align}

This shows that $\Ecal_{x_0}(0)=0$. Since $t\in[0,1]$, $\Ecal_{x_0}(\lambda t)$ is continuous in $\lambda$ and uniformly equicontinuous in $t$. Hence for $\lambda$ sufficiently small 
\begin{equation}\label{errorestimate}
\Ecal_{x_0}(\lambda t)\geq -\epsilon
\end{equation}
uniformly in $t$. 
With \eqref{errorestimate} and the lower bound on the Ricci curvature, \eqref{Deq1} becomes 
\begin{align}
\frac{\ddot \Dcal_{x_0}(t)}{\Dcal_{x_0}(t)}=-\frac{\lambda^2}{N}\Ric^{N,\m}(V(x_0),V(x_0))+\frac{\lambda^2}{N}\Ecal_{x_0}(\lambda t)\geq \frac{\lambda^2}{N}\left(-(K-2\epsilon)-\epsilon\right)=-\frac{(K-\epsilon)\lambda^2}{N},
\end{align}
which gives the differential inequality
\begin{align}
\ddot \Dcal_{x_0}(t)+\frac{(K-\epsilon)\lambda^2}{N}\Dcal_{x_0}(t)>0,
\end{align}
as desired.
\end{proof}

Having obtained a convexity condition for the infinitesimal distortion, we now `back off' the limit to obtain a weaker convexity condition along optimal interpolation. 
\begin{lemma}\label{InteqLemma}
Assume the premise of Lemma \ref{DeqLemma}. Let $\{A_\delta\}_{\delta>0}$ be a nice family of sets that shrink to $x_0$ as $\delta\to 0$. There exists constants $C_1,C_2>0$, $\delta_0$ and $\lambda_0$ such that for $\delta<\delta_0$, $\lambda <\lambda_0$, 
\begin{align}
\m^{1/N}(T_{\lambda t}(A_\delta))&\leq \tau_{K,N}^{(1-t)}(\Theta)\m^{1/N}(A_\delta)+\tau_{K,N}^{(t)}(\Theta)\m^{1/N}(T_\lambda(A_\delta))
+\left(C_1(\delta+\lambda^4)-C_2\lambda^2\right)\delta^{n/N}.
\end{align}
\end{lemma}
\begin{proof}
As a result of Lemma \ref{DeqLemma} and Theorem \ref{theorem1}, we have that 
\[\Dcal_{x_0}((1-t)a+tb)\leq \sigma_{\frac{K-\epsilon}{N}}^{(1-t)}(\lambda(b-a))\Dcal_{x_0}(a)+ \sigma_{\frac{K-\epsilon}{N}}^{(t)}(\lambda(b-a))\Dcal_{x_0}(b),\]
for any $0\leq a\leq b\leq 1$. In particular, with $a=0,b=1$, we have 
\begin{align}\label{inteq}
\Dcal_{x_0}(t)\leq \sigma_{\frac{K-\epsilon}{N}}^{(1-t)}(\lambda)\Dcal_{x_0}(0)+ \sigma_{\frac{K-\epsilon}{N}}^{(t)}(\lambda)\Dcal_{x_0}(1).
\end{align}

One can easily check from the definition $\sigma_k^{(t)}(\theta)\coloneqq \frac{\sin_k(t\theta)}{\sin_k(\theta)}$, valid for all $k\theta^2<\pi^2$, that, as $\theta\to 0$, 
\begin{align*}
\sigma_{k}^{(t)}(\theta)= t+t(1-t)(1+t)\frac{k}{6}\theta^2+O(\theta^4).
\end{align*}
 We thus deduce that as $\lambda \to 0$,
\begin{align}
\sigma_{\frac{K-\epsilon}{N}}^{(t)}(\lambda)= \sigma_{\frac{K}{N}}^{(t)}(\lambda)-t(1-t^2)\frac{\epsilon}{6N}\lambda^2+O(\lambda^4),
\end{align}

Therefore equation \eqref{inteq} implies that as $\lambda \to 0$, 
\begin{align}\label{inteq2}
\Dcal_{x_0}(t)\leq \sigma_{\frac{K}{N}}^{(1-t)}(\lambda)\Dcal_{x_0}(0)+\sigma_{\frac{K}{N}}^{(t)}(\lambda)\Dcal_{x_0}(1)- t(1-t)\frac{\epsilon}{6N}\Big((2-t)\Dcal_{x_0}(0)+(1+t)\Dcal_{x_0}(1)\Big)\lambda^2+O(\lambda^4).
\end{align} 

 Continuity and positivity of $\Dcal_{x_0}$ over the compact interval $[0,\lambda]$ allows us to find $M>0$ to bound the entire quantity: 
\[(2-t)\Dcal_{x_0}(0)+(1+t)\Dcal_{x_0}(1)>M>0,\]
so that \eqref{inteq2} can be more compactly written as
\begin{align}
\Dcal_{x_0}(t)\leq \sigma_{\frac{K}{N}}^{(1-t)}(\lambda)\Dcal_{x_0}(0)+\sigma_{\frac{K}{N}}^{(t)}(\lambda)\Dcal_{x_0}(t)- t(1-t)\frac{M\epsilon}{6N}\lambda^2+O(\lambda^4),
\end{align} 
as $\lambda \to 0$. Since the sets $A_\delta$ shrink to $x_0$, for sufficiently small $\delta$ and $\lambda$, the product $A_\delta\times T_\lambda(A_\delta)\subset \operatorname{sing(\ell)}^\mathsf{c}$, so $\ell(x,y)$ is differentiable on this set. This gives that, as $\delta\to 0$, 
\[\Theta(A_\delta, T_{\lambda}(A_\delta))= \ell^+(x_0, T_\lambda(x_0))+O(\delta)= \lambda+O(\delta),\]
so we have 
\begin{align}\label{integralDeqcont*}
\Dcal_{x_0}(t)\leq \sigma_{\frac{K}{N}}^{(1-t)}(\Theta)\Dcal_{x_0}(0)+\sigma_{\frac{K}{N}}^{(t)}(\Theta)\Dcal_{x_0}(1)- t(1-t)\frac{M\epsilon}{6N}\lambda^2+O(\lambda^4)+O(\delta),
\end{align}
as $\delta\to 0$. 
Next, by definition of $\Dcal_{x_0}(t)$, 
\[\frac{\m^{1/N}(T_{\lambda t}(A_\delta))}{\m^{1/N}(A_\delta)}=\Dcal_{x_0}(t)+O(\delta).\]

Since $\m^{1/N}(A_\delta)\in O(\delta^{n/N})$ this implies that 
\[\m^{1/N}(A_\delta)\Dcal_{x_0}(t)= \m^{1/N}(T_{\lambda t}(A_\delta))+O(\delta^{1+n/N}).\]

Therefore multiplying \eqref{integralDeqcont*} by $\m^{1/N}(A_\delta)$ and bounding the necessary quantities gives that, as $\delta,\lambda\to 0$, 
\begin{align}
\begin{split}
\m^{1/N}(T_{\lambda t}(A_\delta))&\leq \sigma_{K/N}^{(1-t)}(\Theta)\m^{1/N}(A_\delta)+\sigma_{K/N}^{(t)}(\Theta)\m^{1/N}(T_\lambda (A_\delta))\\
&+\left(-t(1-t)\frac{M\epsilon}{6N}\lambda ^2+ O(\lambda^4)+O(\delta)\right)O(\delta^{n/N}).
\end{split}
\end{align}
Finally we use the simple fact that $\sigma_{K/N}^{(t)}(\Theta)\leq \tau_{K,N}^{(t)}(\Theta)$. We conclude that there exists constants $C_1,C_2$ such that for $\lambda, \delta$ sufficiently small, 
\begin{align}
\begin{split}
\m^{1/N}(T_{\lambda t}(A_\delta))&\leq \tau_{K,N}^{(1-t)}(\Theta)\m^{1/N}(A_\delta)+\tau_{K,N}^{(t)}(\Theta)\m^{1/N}(T_\lambda(A_\delta))
+\left(C_1(\delta+\lambda^4)-C_2\lambda^2\right)\delta^{n/N}.
\end{split}
\end{align}
\end{proof}

\subsection{Optimal and geodesic transport comparison}
Now we compare geodesic and optimal transport. The main result of this section is Lemma \ref{optimaltogeodesic}. 
Fix a geodesically convex normal neighbourhood $U$ of $x_0$, on which the inverse of the exponential map $\Log$ is defined. Let 
\[F_t(x,y)\coloneqq \exp(t\Log(x,y)),\]
for all $x,y\in U$. Notice that $T_{\lambda t}(x)=F_{t }(x,T_\lambda (x))$. It follows that 
\[T_{\lambda t}(A_\delta)\subset F_t(A_\delta\times T_\lambda(A_\delta)).\]

Lemma \ref{optimaltogeodesic} will show that this inclusion is not too large in measure. To do this we need to establish two key estimates.

\begin{prop}\label{DF_s}
Fix $(x_0, x_1)\in M\times M$ in the domain of $\Log$ with $x_0\ll x_1$. Let $x_\lambda$ be the unique timelike geodesic which joins them. Through Fermi coordinates along $x_\lambda$, we may identify the family of tangent spaces $T_{x_\lambda}M$ with $\R^n$. Then, in Fermi coordinates about the geodesic $(x_0, x_\lambda)\in M\times M$, and for any $(u,v)\in \R^n\times \R^n$ we have 
\begin{align*}
(DF_t)_{(x_0, x_\lambda)}(u,v)= (1-t) u+t v-\lambda^2\frac{t(1-t)}{6} R_{\log_{x_0}(x_1)}\Big((2-t)u+(1+t)v\Big)+O(\lambda^3),
\end{align*}
as $\lambda\to 0$. 
\end{prop}
\begin{proof}
The matrix of Jacobi fields $M_\lambda(t):= (DF_t)_{(x,x_\lambda)}$ solves
\begin{align}
\begin{cases}
\ddot M_\lambda(t) -R_{\dot x_{t\lambda}} M_\lambda =0;\\
M_\lambda(0)= \pi_1;\\
M_\lambda(1)= \pi_2,
\end{cases}
\end{align}
where $R_{V}:= \Riem(V,\cdot)V$ for any vector field $V$, and $\pi_i:\R^n\to \R^n\to \R^n$ is the projection onto the $i$-th factor, $i=1,2$. In normal coordinates about $x_0$, we see that $\partial_\lambda x_\lambda:= x'_\lambda = \lambda \log_{x_0}(y_0)$, and so it follows that $R_{\dot x_{t\lambda}}= \lambda^2 \Riem_{\dot x_{t\lambda}}(\log_{x_0}(y_0), \cdot)\log_{x_0}(y_0)$. 

At $\lambda=0$, we have that $M_0(t)=(1-t)\pi_1+t\pi_2$. Taking derivatives with respect to $\lambda$, we see that $M'_0(t):= \partial_\lambda\vert_{\lambda=0} M_\lambda(t)$ solves 
\begin{align}
\begin{cases}
\ddot M_\lambda(t)  =0;\\
M'_0(0)= 0;\\
M'_0(1)= 0.
\end{cases}
\end{align}
Therefore $M'_0(t)\equiv 0$. Moreover $M''_0(t):= \partial_\lambda^2\vert_{\lambda=0}M_\lambda(t)$ solves 
\begin{align}
\begin{cases}
\ddot M''_0(t)  -2R_{\dot x_0} M_0(t)=0;\\
M''_0(0)= 0;\\
M''_0(1)= 0.
\end{cases}
\end{align}
So  $M''_0(t) = -\frac{t(1-t)}{3}R_{\dot x_0}((2-t)\pi_1+(1+t)\pi_2)$. 

It follows that 
\[M_\lambda(t) =(1-t)\pi_1+t\pi_2-\lambda^2\frac{t(1-t)}{6}R_{\dot x_0}((2-t)\pi_1+(1+t)\pi_2)+O(\lambda^3),\]
as desired.
\end{proof}

\begin{prop}\label{Tballestimate}
Let $\rho = \rho(\delta)$ be some polynomial function of $\delta$. Let $\bar g$ be the background Riemannian metric which induces the topology on $M$. Then, in normal coordinates about $x_0$, and for sufficiently small $\lambda$, 
\[\m^{1/N}(B^{\bar g}_{\delta \rho}(T_{\lambda t}(A_\delta))) = \m^{1/N}(T_{\lambda t}(A_\delta)) + O(\delta^{n/N}(\delta+\rho)),\]
as $\delta\to 0$.
\end{prop}
\begin{proof}
For sufficiently small $\delta$, $T_{\lambda t}(A_\delta)$ stays in a normal neighbourhood about $x_0$. In normal coordinates, let $\vol_n$ denote the Lebesgue measure. For sufficiently small $\lambda$, $T_{\lambda t}$ is a local diffeomorphism, and so in particular, $T_\lambda(A_\delta)$ is convex and has bounded eccentricity as $\delta\to 0$. It follows that there exists some $f(x)>0$ such that 
\[\lim_{\delta\to 0}\frac{\m(B^{\bar g}_{\delta \rho}(T_{\lambda t}(A_\delta)))}{\vol_n(B^{\bar g}_{\delta \rho}(T_{\lambda t}(A_\delta)))} = \lim_{\delta\to 0}\frac{\m(T_{\lambda t}(A_\delta))}{\vol_n(T_{\lambda t}(A_\delta))}= f(T_{\lambda t}(x_0)).\]
Recall that $B^{\bar g}_{\delta \rho}(T_{\lambda t}(A_\delta))$ is contained in the normal neighbourhood $U$ about $x_0$. Taking $\delta$ sufficiently small, we may assume it is contained in a compact subset of $U$, and therefore there is a $\Lambda>0$ independent of $\delta, \lambda$, such that the Riemannian ball $B^{\bar g}_{\delta \rho}(T_{\lambda t}(A_\delta))$ is contained in the Euclidean ball $B_{\Lambda\delta \rho}(T_{\lambda t}(A_\delta))$. Steiner's formula (\cite{Schneider2013}, eq.\ 4.8)  gives that 
\[\vol_n(B_{\Lambda\delta \rho}(T_{\lambda t}(A_\delta))) = \sum_{k=0}^n \binom{n}{k} W_k(T_{\lambda t}(A_\delta))(\Lambda\delta\rho)^k, \]
where the values $W_k(T_{\lambda t}(A_\delta))$ are the $k$th quermassintegrals of $T_{\lambda t}(A_\delta)$. We may express the $k$-th quermassintegral as a mixed volume \cite[eq.\ 5.53]{Schneider2013}, 
\[W_k(K)=V(\underbrace{K,\dots, K}_{\text{$n-k$ times}},\underbrace{B_1(0), \dots, B_1(0)}_{\text{$n$ times}}),\]
for any closed convex set with non-empty interior $K$. Linearity of the mixed volume in each argument \cite[eq.\ 5.17]{Schneider2013} then confirms that $W_k(T_{\lambda t}(A_\delta))\in O(\delta^{n-k})$. Moreover, $W_0(T_{\lambda t}(A_\delta))=\vol_n(T_{\lambda t}(A_\delta))$ \cite[eq.\ 5.8]{Schneider2013}. Hence 
\[\vol_n(B_{\Lambda\delta \rho}(T_{\lambda t}(A_\delta))) = \vol_n(T_{\lambda t}(A_\delta))+\sum_{k=1}^n\binom{n}{k}( \Lambda\delta\rho)^k O(\delta^{n-k})= \vol_n(T_{\lambda t}(A_\delta))+O(\delta^{n}\rho),\]
and therefore 
\begin{align*}
\m(B^{\bar g}_{\delta \rho}(T_{\lambda t}(A_\delta)))& = f(T_{\lambda t}(x_0)) \vol_n(B^{\bar g}_{\delta \rho}(T_{\lambda t}(A_\delta)))+O(\delta^{1+n})\\
&= f(T_{\lambda t}(x_0))\vol_n(T_{\lambda t}(A_\delta))+ O(\delta^n(\delta+\rho)) \\
&= \m(T_{\lambda t}(A_\delta))+O(\delta^n(\delta+\rho)).
\end{align*}
Note that $\m(T_{\lambda t}(A_\delta))\sim \delta^n$, meaning it is $O(\delta^n)$ but not $O(\delta^k)$ for any $k<n$. As a result, we have that 
\[\frac{\m(B^{\bar g}_{\delta \rho}(T_{\lambda t}(A_\delta)))}{\m(T_{\lambda t}(A_\delta))} = 1+ O(\delta+\rho)\implies\frac{\m^{1/N}(B^{\bar g}_{\delta \rho}(T_{\lambda t}(A_\delta)))}{\m^{1/N}(T_{\lambda t}(A_\delta))} = 1+ O(\delta+\rho). \]

Again, $\m^{1/N}(T_{\lambda t}(A_\delta)) \in O(\delta^{n/N})$, and so we obtain
\[\m^{1/N}(B^{\bar g}_{\delta \rho}(T_{\lambda t}(A_\delta))) = \m^{1/N}(T_{\lambda t}(A_\delta))+ O(\delta^{n/N}(\delta+\rho)).\]
\end{proof}
We are now able to prove the lemma at the crux of this section. We remind the reader that $A_\delta$ is the image under $\exp_{x_0}$ of a cube of sidelength $\delta$ spanned by the eigenvectors of $R_{v_0}\coloneqq\Riem_{x_0}(v_0, \cdot )v_0$. 
\begin{lemma}\label{optimaltogeodesic}
Fix $t\in[0,1]$. There exists $C>0$ and $\lambda_0, \delta_0$ such that for all $\lambda<\lambda_0$, $\delta<\delta_0$,
\[\m^{1/N}(F_t(A_\delta, T_t(A_\delta))) \leq \m^{1/N}(T_{\lambda t}(A_\delta))+ C(\delta+ \lambda^3)\delta^{n/N}.\]
\end{lemma}
\begin{proof}

Fix $t\in [0,1]$. If $\ell(x,y)$ is sufficiently small, and $\lambda\ll1$, then $F_t(x,T_\lambda(y))$ stays in $U$, and therefore we can apply Taylor's theorem to $F_t(x,T_\lambda(y))$ in normal coordinates about $x_0\in U$. In such coordinates we write $x_0=0$, and $x_\lambda := T_\lambda(x_0)$. We have
\begin{align*}
\begin{split}
F_t(x,T_\lambda(y))&= F_t(0,x_\lambda)+ (DF_t)_{(0, x_\lambda)}(x, DT_\lambda(y))+ O\left(\big(|x|+|y|\big)^2\right)\\
&=T_{\lambda t}(0)+  (DF_t)_{(0,x_\lambda)}(x, DT_\lambda(y))+ O\left(\big(|x|+|y|\big)^2\right).
\end{split}
\end{align*}
We want to stress two things at this time. The first is that $t$ is fixed, so none of the higher order terms depend on $t$. The second is that the expansion is done relative to the Riemannian metric $\bar g$ which induces the topology. That is, $|x|$ measures the {\it Riemannian} length of $x$, not the Lorentzian length.

Recall that $DT_\lambda$ is the solution map for an initial value Jacobi problem, and so $(DT_\lambda)_x= \Id+\lambda (DV)_x+\frac{\lambda^2}{2}R_{V(x)}+O(\lambda^3)$. In particular 
\begin{align}\label{DT_t}
(DT_\lambda)_{x_0}= (1+\alpha\lambda )\Id +\frac{\lambda^2}{2}R_{v_0}+O(\lambda^3).
\end{align}
 Combining equation \eqref{DT_t} with proposition \ref{DF_s}, we deduce that 
\begin{align*}
\begin{split}
(DF_t)_{(x_0, x_\lambda)}(x, (DT_\lambda)_{x_0} y)=&(1-t)x+t(1+\alpha\lambda)y+\frac{t\lambda^2}{2}R_{v_0}y\\
&-\frac{t(1-t)\lambda^2}{6}R_{v_0}\Big((2-t)x+(1+t)(1+\alpha\lambda)y\Big)+O(\lambda^3(|x|+|y|)).
\end{split}
\end{align*}
Therefore, defining $M_1:= (1-t)\Id-\frac{t(1-t)(2-t)\lambda^2 }{6}R_{v_0}$ and $M_2:= t(1+\alpha \lambda)\Id+\left(\frac{t\lambda^2}{2}-\frac{t(1-t)(1+t)(1+\alpha\lambda)\lambda^2}{6}\right)R_{v_0}$, we see that 
\begin{align}
F_t(x,T_\lambda(y))= T_{\lambda t}(0)+M_1x+M_2y+O\Big(\big(|x|+|y|+\lambda^3\big)\big(|x|+|y|\big)\Big).
\end{align}
  
Since $T_{\lambda t}(x)= F_t(x, T_\lambda(x))$, we then have 
\begin{align}\label{Texpansion}
T_{\lambda t}(x)= T_{\lambda t}(0)+(M_1+M_2)x+O\Big((|x|+\lambda^3)|x|\Big).
\end{align}

Notice that each operator $M_1,M_2, M_1+M_2$ are linear combinations of $\Id$ and $R_{v_0}$, and therefore they can be simultaneously diagonalized by the eigenvectors of $R_{v_0}$. Now, in normal coordinates about $x_0$, $A_\delta$ is a cube of length $\delta$ spanned by the eigenvectors of $R_{v_0}$; it follows that 
\[M_1A_\delta+M_2A_\delta=(M_1+M_2)A_\delta.\]
Or said another way, for every $x, y\in A$, there is $z$ in $A$ such that $(M_1+M_2)z= M_1x+M_2y$. With this choice of $z$, notice that 
\[|F_t(x,T_\lambda(y))- T_{\lambda t}(z)|= O\Big(\big(|x|+|y|+\lambda^3\big)\big(|x|+|y|\big)+(|z|+\lambda^3)|z|\Big),\]
as $\lambda\to 0$ and $|x|+|y|+|z|\to 0$. 
Hence there exists $C>0$, $\lambda_0, \delta_0$ such that for $\lambda<\lambda_0$, $\delta<\delta_0$, 
\[|F_t(x,T_\lambda(y))-T_{\lambda t}(z)|\leq C(\delta+\lambda^3)\delta.\]
We write $\rho\coloneqq C(\delta+\lambda^3)$ and conclude that 
\begin{align}\label{Fball}
F_t(A_\delta\times T_\lambda(A_\delta))\subset B^{\bar g}_{\delta\rho}(T_{\lambda t}(A_\delta)).
\end{align}
Applying proposition \ref{Tballestimate}, we see that 
\begin{align*}
\m^{1/N}(F_t(A_\delta\times T_\lambda(A_\delta)))\leq \m^{1/N}( B^{\bar g}_{\delta \rho}(T_{\lambda t}(A_\delta))) =\m^{1/N}(T_{\lambda t}(A_\delta))+O(\delta^{n/N}(\delta+\rho)).
\end{align*}
Therefore there is $C$ such that for sufficiently small $\lambda$ and $\delta$, 
\begin{align}
\m^{1/N}(F_t(A_\delta\times T_\lambda(A_\delta)))\leq \m^{1/N}(T_{\lambda t}(A_\delta))+C(\delta+\lambda^3)\delta^{n/N},
\end{align}
as desired.
\end{proof}

\subsection{Proof of main theorem}
We now prove the main theorem. 
\begin{proof}[Proof of Theorem \ref{maintheorem}]
 For sufficiently small $\lambda,\delta$, Lemma \ref{InteqLemma} gives 
\begin{align}
\m^{1/N}(T_{\lambda t}(A_\delta))&\leq \tau_{K/N}^{(1-t)}(\Theta)\m^{1/N}(A_\delta)+\tau_{K/N}^{(t)}(\Theta)\m^{1/N}(T_\lambda(A_\delta))
+\left(C_1(\delta+\lambda^4)-C_2\lambda^2\right)\delta^{n/N},
\end{align}
whereas Lemma \ref{optimaltogeodesic} gives 
\begin{align}
\m^{1/N}(F_t(A_\delta\times T_\lambda(A_\delta))) \leq  \m^{1/N}(T_{\lambda t}(A_\delta))+C(\delta+\lambda^3)\delta^{n/N}.
\end{align}
Taking the maximum of $C$ and $C_1$, and relabeling, we obtain:
\begin{align*}
\begin{split}
\m^{1/N}(F_t(A_\delta\times T_\lambda(A_\delta))&\leq \tau_{K,N}^{(1-t)}(\Theta)\m^{1/N}(A_\delta)+\tau_{K,N}^{(t)}(\Theta)\m^{1/N}(T_\lambda(A_\delta))+\left(C_1(\delta+\lambda^3)-C_2\lambda^2\right)\delta^{n/N}.
\end{split}
\end{align*}
With $\lambda\ll 1$ and $\delta\sim \lambda^3$ we may have $C_1(\delta+\lambda^3)-C_2\lambda^2<0$, from which we can conclude that
\begin{align*}
\m^{1/N}(F_t(A_\delta\times T_\lambda(A_\delta))&< \tau_{K,N}^{(1-t)}(\Theta)\m^{1/N}(A_\delta)+\tau_{K,N}^{(t)}(\Theta)\m^{1/N}(T_\lambda(A_\delta)).
\end{align*}
Recalling that $G_t(A_\delta, T_\lambda(A_\delta))\subset F_t(A_\delta\times T_\lambda(A_\delta))$, we finally achieve the desired inequality
\begin{align}
\m^{1/N}(G_t(A_\delta,T_\lambda(A_\delta))&< \tau_{K,N}^{(1-t)}(\Theta)\m^{1/N}(A_\delta)+\tau_{K,N}^{(t)}(\Theta)\m^{1/N}(T_\lambda(A_\delta)).
\end{align}
\end{proof}

\printbibliography[heading=bibintoc]
\end{document}